\documentclass{amsart}
\usepackage{tikz,tikz-cd}
\usepackage{url}
\usepackage{graphicx}
\usepackage{color}
\usepackage{times}
\usepackage{cite}
\usepackage{enumerate,latexsym}
\usepackage{latexsym}
\usepackage{amsmath,amssymb}
\usepackage{graphicx}
\usepackage{amsthm}
\usepackage{verbatim}
\newtheorem{thm}{Theorem}[section]
\newtheorem*{theorem*}{Theorem}
\newtheorem*{acknowledgement*}{Acknowledgement}

\newtheorem{lem}[thm]{Lemma}
\newtheorem{prop}[thm]{Proposition}
\theoremstyle{definition}

\theoremstyle{remark}
\newtheorem{rem}[thm]{Remark}

\numberwithin{equation}{section}

\newcommand{\set}[1]{\left\{#1\right\}}
\newcommand{\Real}{\mathbb R}

\newcommand{\func}[1]{\ensuremath{\mathop{\mathrm{#1}}} }

\newcommand{\Div}[0]{\func{div}}

\newcommand{\dist}[0]{\mathrm{dist}}

\newcommand{\xX}[0]{\mathbf{x}}

\newcommand{\yY}[0]{\mathbf{y}}

\newcommand{\nN}[0]{\mathbf{n}}

\newcommand{\OO}{\mathbf{0}}

\title[Hausdorff Stability of the Round Two-Sphere]{Hausdorff Stability of the Round Two-Sphere Under Small Perturbations of the Entropy}
\author{Jacob Bernstein}
\address{Department of Mathematics, Johns Hopkins University, 3400 N. Charles Street, Baltimore, MD 21218}
\email{bernstein@math.jhu.edu}
\author{Lu Wang}
\address{Department of Mathematics, University of Wisconsin, 480 Lincoln Drive, Madison, WI 53706}
\email{luwang@math.wisc.edu}
\thanks{The first author was partially supported by the NSF Grant DMS-1307953. The second author was partially supported by the NSF Grant DMS-1406240. This material is based upon work supported by the NSF Grant DMS-1440140 while the second author was in residence at the Mathematical Sciences Research Institute (MSRI) in Berkeley, CA, during the Spring 2016 semester.}

\begin{document}
\begin{abstract}
We show that if a closed surface in $\Real^3$ has entropy near to that of the unit two-sphere, then the surface is close to a round two-sphere in the Hausdorff distance. 
\end{abstract}

\maketitle

\section{Introduction} \label{Intro}
If $\Sigma$ is a {hypersurface}, that is, a connected, smooth, properly embedded, codimension-one submanifold of $\Real^{n+1}$, then its 
\emph{Gaussian surface area} is
\begin{equation*}
F[\Sigma]=\int_{\Sigma}\Phi\, d\mathcal{H}^{n}=(4\pi)^{-\frac{n}{2}}\int_{\Sigma}e^{-\frac{|\xX|^2}{4}} d\mathcal{H}^n,
\end{equation*}
where $\mathcal{H}^n$ is $n$-dimensional Hausdorff measure on $\Real^{n+1}$. Following Colding-Minicozzi \cite{CMGenMCF}, its \emph{entropy} is
\begin{equation*}
\lambda[\Sigma]=\sup_{(\yY,\rho)\in\Real^{n+1}\times\Real^+}F[\rho\Sigma+\yY].
\end{equation*}
That is, the entropy of $\Sigma$ is the supremum of the Gaussian surface area over all translations and dilations of $\Sigma$. Observe that the entropy of a hyperplane is one. 

In \cite{BernsteinWang}, we show that, for $2\leq n\leq 6$, the entropy of a closed hypersurface in $\Real^{n+1}$ is uniquely (modulo translations and dilations) minimized by $\mathbb{S}^n$, the unit $n$-sphere (centered at the origin). This verifies a conjecture of Colding-Ilmanen-Minicozzi-White \cite[Conjecture 0.9]{CIMW} (cf. \cite{KetoverZhou, JZhu} for related interesting results). We further show, in \cite[Corollary 1.3]{BernsteinWang2}, that surfaces in $\Real^3$ of small entropy are topologically rigid. That is, if $\Sigma$ is a closed surface in $\Real^3$ and $\lambda[\Sigma]\leq\lambda[\mathbb{S}^1\times\Real]$, then $\Sigma$ is diffeomorphic to $\mathbb{S}^2$.   

In this short note we show that a closed surface in $\Real^3$ whose entropy is near $\lambda_2=\lambda[\mathbb{S}^2]$ must be close to some round sphere in (normalized) Hausdorff distance.
\begin{thm}
\label{MainThm} 
Given $\epsilon>0$, there is a $\delta=\delta(\epsilon)>0$ so that if $\Sigma\subset \Real^3$ is a closed surface with $\lambda[\Sigma]< \lambda[\mathbb{S}^2]+\delta$, then $\dist_{H}(\Sigma,  \rho\mathbb{S}^2+\yY)<   \rho\epsilon$ for some $\rho>0$ and $\yY\in \Real^3$. 
\end{thm} 

This result may appear somewhat surprising as surfaces obtained by gluing long thin ``spikes" to $\mathbb{S}^2$ have Gaussian surface area arbitrarily close to that of $\mathbb{S}^2$ while being quite far from any round sphere in the Hausdorff distance. The key distinction is that, unlike the Gaussian surface area, the entropy ``sees" all scales. Indeed, heuristically one would expect the presence of such a ``spike" to cause the entropy to be at least $\lambda_1=\lambda[\mathbb{S}^1\times \Real]>\lambda[\mathbb{S}^2]$ as on some scale the spikes should (qualitatively) look like a  cylinder.

In order to make this intuition rigorous, we use the classification of two-dimensional self-shrinkers of low entropy provided by \cite{BernsteinWang2}.  Specifically, we obtain a variant on a curvature estimate of White \cite{WhiteReg}. Roughly speaking, we show that as long as the entropy of each surface in a mean curvature flow is strictly below $\lambda_1$ and the flow remains smooth on the full parabolic cylinder about a space-time point on the flow, then the curvature at the point is controlled by the inverse of the parabolic radius of the cylinder.  This may be interpreted as a ``speed" bound and is what rules out the presence of ``spikes" (as the tips would have to move rapidly).  As this result may be of independent interest we record it here.
\begin{thm}\label{RefinedWhiteThm}
Given $\epsilon>0$, there is a constant $C=C(\epsilon)$ so that if $S=\set{\Sigma_t}_{t\in I}$ is a smooth mean curvature flow in an open subset $U\subset \Real^3$ with $\lambda[S]\leq \lambda_1-\epsilon$, then for each $p\in \Sigma_t$ satisfying $C_r(\xX(p),t)\subset U\times I$, where $\xX$ is the position vector and
$$
C_r(\xX(p),t)=\set{(\yY,s)\in\Real^{3}\times\Real : |\yY-\xX(p)|<r, |s-t|<r^2} ,
$$ 
it follows that
$$
|A_{\Sigma_t}(p)|\leq C r^{-1}.
$$ 
\end{thm}

Finally, we observe that Theorem \ref{MainThm} can be sharpened by giving an explicit relationship between (normalized) Hausdorff distance to a round sphere and the difference between the entropy of the surface and the entropy of the round sphere.
\begin{thm}\label{HolderThm}
There exists a universal constant $K>0$ so that: If $\Sigma$ is a closed surface in $\Real^3$, then there exists a $\rho>0$ and $\yY\in \Real^3$ so that
$$
\dist_H(\Sigma, \rho \mathbb{S}^2+\yY)\leq K \rho \left(\lambda[\Sigma]-\lambda[\mathbb{S}^2]\right)^{\frac{1}{8}}.
$$
\end{thm}
\begin{rem}
The exponent $\frac{1}{8}$ may not be sharp.  
\end{rem}
The proof of this result is more technical than that of Theorem \ref{MainThm} as it makes more direct use of parabolic estimates.  For this reason we defer it to the end of the paper.  

As a final remark, we observe that Theorem \ref{HolderThm} may be thought of as a Bonnesen-style inequality for the entropy.
Recall, Osserman \cite{Osserman} considers a Bonnesen-style {isoperimetric} inequality to be an inequality bounding the isoperimetric defect of a planar curve $\beta$ from below by a non-negative measure of ``roundness".  The prototypical example is the classical inequality of Bonnesen \cite{Bonnesen}
$$
 \pi^2 \left(R_{out}-R_{in}\right)^2\leq \mathbf{L}-4\pi \mathbf{A}.
$$
Here $\mathbf{L}$ and $\mathbf{A}$ are the length and enclosed area of $\beta$ while $R_{out}$ is the circumradius, the smallest radius of any circle containing $\beta$, and $R_{in}$ is the inradius, the largest radius of any circle enclosed by $\beta$.
 The connection to Theorem \ref{MainThm} is made clearer by noting that $R_{out}-R_{in}$ controls the Hausdorff distance between $\beta$ and some round circle. 

\section{Notation}
Denote a (open) ball in $\Real^{n}$ of radius $R>0$ and center $\xX$ by $B^n_R(\xX)$ and the closed ball by $\bar{B}^n_R(\xX)$. We often omit the superscript $n$ when there will be no confusion. Likewise, denote the (open) parabolic cylinder centered at $(\xX,t)$ of parabolic radius $R$ by 
$$
C_R(\xX,t)=B_R(\xX)\times(t-R^2,t+R^2)
$$
and the (open) past parabolic cylinder by
$$
C_R^-(\xX,t)=B_R(\xX)\times(t-R^2,t).
$$
Given two compact subsets $X,Y\subset \Real^n$, the \emph{Hausdorff distance} $\dist_H(X,Y)$ between $X$ and $Y$ is defined by
$$
\dist_H(X,Y)=\inf\set{r>0: X\subset \bigcup_{\xX\in Y} \bar{B}_r(\xX) \mbox{ and } Y\subset \bigcup_{\xX\in X} \bar{B}_r(\xX)}.
$$

Fix an open subset $U\subset \Real^{n+1}$.  A \emph{hypersurface in $U$}, $\Sigma$, is a smooth, properly embedded, codimension-one connected submanifold of $U$. At times it is convenient to distinguish between a point $p\in\Sigma$ with its \emph{position vector} $\xX(p)$.
A \emph{smooth mean curvature flow in $U$}, $S$, is a collection of hypersurfaces in $U$, $\set{\Sigma_t}_{t\in I}$, $I$ an interval, so that
\begin{enumerate}
\item For all $t_0\in I$ and $p_0\in \Sigma_{t_0}$, there is a $r_0=r_0(p_0,t_0)$ and an interval $I_0=I_0(p_0,t_0)$ with $(\xX(p_0),t_0)\in B_{r_0}^{n+1}(p_0)\times I_0\subset U\times I$;
\item There is a smooth map $\Psi: B_{1}^n\times I_0\to \Real^{n+1}$
so $\Psi_t(p)=\Psi(p,t): B_1^n\to \Real^{n+1}$ is a parameterization of $B_{r_0}^{n+1}(p_0)\cap \Sigma_t$; and
\item $\left(\frac{\partial}{\partial t} \Psi (p,t)\right)^\perp= \mathbf{H}_{\Sigma_t}(\Psi (p,t)).$
\end{enumerate}
Here $\mathbf{H}_{\Sigma}=-H_{\Sigma} \nN_{\Sigma}= -\Div_{\Sigma} (\nN_{\Sigma})\nN_{\Sigma}$ is the mean curvature vector of a hypersurface $\Sigma$.
It is convenient to consider the \emph{space-time track} of $S$ (also denoted by $S$):
\begin{equation*}
S=\set{(\xX(p),t)\in\mathbb{R}^{n+1}\times \mathbb{R}: p\in\Sigma_t}\subset U\times I.
\end{equation*}
This is a smooth (possibly with boundary) submanifold of space-time and is transverse to each constant time hyperplane. 
We define the \emph{parabolic rescaling} of $S$ about $(\xX,t)\in\Real^{n+1}\times\Real$ by $\rho>0$ to be
$$
\rho (S-(\xX,t))=\set{(\yY,s)\in\Real^{n+1}\times\Real : (\rho^{-1}\yY+\xX_0, \rho^{-2}s+t)\in S}.
$$

For a hypersurface, $\Sigma$, in an open set $U$, we extend the definitions of $F$ and $\lambda$ in an obvious way, namely,
$$
F[\Sigma]= (4\pi)^{-n} \int_{\Sigma} e^{-\frac{|\xX|^2}{4}} d\mathcal{H}^n \mbox{ and } \lambda[\Sigma] =\sup_{(\yY,\rho)\in \Real^{n+1}\times\Real^+} F[\rho \Sigma+\yY]
$$
so that these agree with the standard definition when $U=\Real^{n+1}$.
If $\Sigma$ is a hypersurface in $U$ and $V$ is an open subset of $U$, then any component of $V\cap \Sigma$ is a hypersurface in $V$ and $F[\Sigma\cap V]\leq F[\Sigma]$ and $\lambda[\Sigma\cap V]\leq \lambda [\Sigma]$.
Given a smooth mean curvature flow in $U$, $S=\set{\Sigma_t}_{t\in I}$, define the \emph{entropy of the flow $S$} by 
$$
\lambda[S]=\sup_{t\in I} \lambda[\Sigma_t].
$$
It follows from Huisken's monotonicity formula \cite{Huisken} that if $U=\Real^{n+1}$ and $I=[T_1, T_2)$, then $\lambda[S]=\lambda[\Sigma_{T_1}]$.

A hypersurface $\Sigma$ in $\Real^{n+1}$ is a \emph{self-shrinker} if 
$$
\mathbf{H}_\Sigma+\frac{\xX^\perp}{2}=0
$$
where $\xX^\perp$ is the normal component of the position vector.
This is equivalent to 
$$
S(\Sigma)=\set{\sqrt{-t}\, \Sigma}_{t\in (-\infty,0)}
$$
being a smooth mean curvature flow in $\Real^{n+1}$.
By \cite[Lemma 7.10]{CMGenMCF} and Huisken's monotonicity formula
$$
F[\Sigma]=\lambda[\Sigma]=\lambda[S(\Sigma)].
$$
Important examples are the self-shrinking cylinders defined, for $0 \le k \le n$, by
$$
(\sqrt{2k}\, \mathbb{S}^{k})\times \Real^{n-k}=\set{(\xX,\yY)\in\Real^{k+1}\times\Real^{n-k} : |\xX|^2=2k}.
$$
One has 
$$
\lambda_k=\lambda[\mathbb{S}^k]=F[\sqrt{2k}\, \mathbb{S}^k]=\lambda[\mathbb{S}^{k}\times\Real^{n-k}]=F[(\sqrt{2k}\, \mathbb{S}^{k})\times \Real^{n-k}]
$$
and by Stone \cite{Stone},
$$
2>\lambda_1>\frac{3}{2}>\lambda_2 > \cdots > \lambda_n> \cdots\to \sqrt{2}.
$$

\section{Sharp Variant of White's Curvature Estimate}
In \cite{WhiteReg}, White gave  an elementary proof of a curvature estimate for smooth mean curvature flows of small entropy. This estimate can be thought of as a version of Brakke's regularity theorem \cite{B} which holds for a restricted (but still large) class of (possibly singular) mean curvature flows.  One of the key ingredients in White's proof was Huisken's monotonicity formula something that was unavailable to Brakke (see \cite{KT} for an approach using the monotonicity formula to prove Brakke's original result). Theorem \ref{RefinedWhiteThm} is a variant of White's estimate.

In order to clarify the relationship between the two estimates, we first observe that White's estimate may be formulated as follows: 
\begin{thm} \label{WhiteThm}
There exists an $\bar{\epsilon}=\bar{\epsilon}(n)$ and a $\bar{C}=\bar{C}(n)$ so that if $S=\set{\Sigma_t}_{t\in I}$ is a smooth mean curvature flow in $U$ with $\lambda[S]\leq 1+\bar{\epsilon}$, then, for each $(p,t)$ satisfying that $p\in \Sigma_t$ and $C_r^-(\xX(p),t)\subset U\times I$,
$$
|A_{\Sigma_t}(p)|\leq \bar{C} r^{-1}.
$$ 
\end{thm}
A consequence (cf. \cite[Proposition 3.2]{WhiteReg}) of White's argument is that the infimum, $\bar{\epsilon}_{0}=\bar{\epsilon}_0(n)$, for which the theorem does not hold for any constant $\bar{C}$ is characterized by the existence of a smooth, non-flat ancient solution to the mean curvature flow in $\Real^{n+1}$ with entropy $1+\bar{\epsilon}_0$.  By Huisken's monotonicity formula this is equivalent (modulo certain technical regularity issues) to the existence of a non-flat self-shrinker of entropy $1+\bar{\epsilon}_0$.  

In \cite[Corollary 1.2]{BernsteinWang2}, it was shown that the only self-shrinkers in $\Real^3$ with entropy less than $\lambda_1$ are the static planes and the shrinking sphere $2 \mathbb{S}^2$. In particular, the best possible $\bar{\epsilon}(2)$ is $\bar{\epsilon}_0(2)=\lambda_2-1$ and this can easily be seen to be sharp by considering points $(\xX(p),t)$ on the flow associated to $2\mathbb{S}^2$ with $t$ sufficiently small.  

Observe that one of the hypotheses of Theorem \ref{WhiteThm} is that the flow is smooth in the backwards parabolic cylinder $C^-_r(\xX(p),t)$.  In Theorem \ref{RefinedWhiteThm} one has the stronger hypothesis that the flow is smooth in the full parabolic cylinder, which allows one to weaken the hypothesis on the entropy. The distinction is that in this case, the sharp bound for the entropy will come from self-shrinkers which are non-compact.  In $\Real^3$ these must have entropy at least $\lambda_1$, but it is unknown what happens in dimension $\geq 4$. 

Before proving Theorem \ref{RefinedWhiteThm} we will need two preliminary results. The first is a classification, in $\Real^3$, of eternal solutions of the mean curvature flow of small entropy.
\begin{prop} \label{EternalSolnProp}
Suppose that $S=\set{\Sigma_t}_{t\in\Real}$ is a non-trivial smooth mean curvature flow in $\Real^3$ with $\lambda[S]<\lambda_1$, then each $\Sigma_t=P$ where $P$ is some plane.
\end{prop}
\begin{rem}
This result is sharp in that  the unique family of rotationally symmetric convex translating solutions to the flow -- i.e., the  bowls \cite{AltschulerWu} --  are eternal solutions with entropy equal to $\lambda_1$ (see \cite{Qiang}).  The bowls can also be used to see that Theorem \ref{RefinedWhiteThm} is sharp.
\end{rem}
\begin{proof}
For each $\rho>0$, let $\rho S=\set{\rho \Sigma_{\rho^{-2} t}}_{t\in\Real}$ be the flow obtained by parabolically rescaling around the space-time origin by $\rho$. Clearly, $\lambda[\rho S]=\lambda[S]$.  Let $\mathcal{S}$ and $\rho \mathcal{S}$ be the associated Brakke flows.  That is, $\mathcal{S}=\set{\mu_{t}}_{t\in\Real}$ where $\mu_{t}=\mathcal{H}^2 \lfloor \Sigma_t$ and the $\rho \mathcal{S}$ are defined analogously.  For a discussion of Brakke flows the reader may consult \cite{IlmanenBook}.  As each $\Sigma_t$ is a closed surface in $\Real^{3}$, the Brakke flows $\rho \mathcal{S}$ are cyclic mod $2$ in the sense of \cite[Definition 4.1]{WhiteMod2}. 

The entropy bound and Brakke's compactness theorem \cite[Section 7.1]{IlmanenBook}, imply that there is a sequence, $\rho_i\to +\infty$, so that $\rho^{-1}_i \mathcal{S}$ converges (in the sense of Brakke flows) to a Brakke flow $\mathcal{S}_{\infty}=\set{\nu_{t}}_{t\in\Real}$. Huisken's monotonicity formula \cite{Huisken} (see \cite{Ilmanen2} for extension to Brakke flows) implies that $\mathcal{S}_\infty$ is non-trivial and backwardly self-similar for $t<0$.  In particular, $\nu_{-1}$ is a self-shrinking measure (cf. \cite[(4.1)]{BernsteinWang}) with $\lambda[\nu_{-1}]< \lambda_1$.  Furthermore, by \cite[Theorem 4.2]{WhiteMod2}, as each $\rho^{-1}_i \mathcal{S}$ is cyclic mod $2$, so is $\mathcal{S}_\infty$. Hence,  $\nu_{-1}=\mathcal{H}^2\lfloor {\Gamma}$ where $\Gamma$ is a self-shrinker with $\lambda[\Gamma]<\lambda_1$.  When $\lambda[\nu_{-1}]<\frac{3}{2}$ this follows from \cite[Proposition 4.2]{BernsteinWang} whereas when $\lambda[\nu_{-1}]\in [\frac{3}{2},\lambda_1)$ this follows as in the proof of \cite[Theorem 1.3]{BernsteinWang}.   A straightforward consequence of Theorem \ref{WhiteThm} and interior parabolic estimates \cite{EH} (cf. \cite[Lemma 2.1]{Schulze}) is that $\rho_i^{-1} \Sigma_{-\rho_i^2}\to \Gamma$ in $C^{\infty}_{loc}(\Real^{3})$.

By  \cite[Corollary 1.2]{BernsteinWang2}, $\Gamma$ is either a plane $P$ through the origin or $2\mathbb{S}^2$. If it is the sphere, then by the smooth convergence,  for $i$ sufficiently large, $\rho_i^{-1} \Sigma_{-\rho_i^2}$ is disjoint from $\partial B_{4}$ but meets $B_4$.  That is, $\Sigma_{-\rho_i^2}\cap\partial  B_{4\rho_i}=\emptyset$ and $\Sigma_{-\rho_i^2}\cap B_{4\rho_i} \neq \emptyset$.  Hence, by comparing with the mean curvature flow of $\partial B_{4\rho_i}$, the flow of $\Sigma_{-\rho_i^2}$ must form a singularity in finite time.  However, this contradicts the hypothesis that $S$ is a smooth flow in $\Real^{3}$ for all $t\in\Real$ and so we must have that $\Gamma$ is a plane through the origin.

As such, by the monotonicity formula, $\lambda[S]=\lambda[\Gamma]=1$ and each $\Sigma_t$ is flat by Theorem \ref{WhiteThm} and hence is a (constant) translation of  $\Gamma$.
\end{proof}

We will also need a parabolic analog of a standard blow-up result, and we include a proof for the sake of completeness.
\begin{lem} \label{BlowupLem}
Let $S=\set{\Sigma_\tau}_{\tau\in (t-4R^2, t+4R^2)}$ be a smooth mean curvature flow in $B_{2R}(\xX(p))$ with $p\in \Sigma_t$ and 
$$
|A_{\Sigma_t}(p)| \geq 4 N R^{-1}>0.
$$
There exist $(\xX(q),s)\in S$ and $\gamma>0$ so that $C_{N\gamma} (\xX(q),s)\subset C_{2R}(\xX(p),t)$
and 
$$
\sup_{(\xX(p'),\tau)\in S\cap C_{N \gamma}(\xX(q),s)} |A_{\Sigma_\tau}(p')|\leq 2 \gamma^{-1} = 2|A_{\Sigma_s}(q)|.
$$
\end{lem}
\begin{proof}
Let $\rho(\yY,\tau)=\max \set{ |\yY-\xX(p)|, |\tau-t|^{1/2}}$ and define a continuous function $f$ on the restriction of $S$ to $\bar{C}_{ R}(\xX(p),t)=\bar{B}_R(\xX(p))\times[t-R^2, t+R^2]$, the closed parabolic cylinder, by
$$
f(\xX(p'),\tau)= \left(R-\rho(\xX(p'),\tau)\right) |A_{\Sigma_\tau}(p')|.
$$
As $f$ is continuous and vanishes on the boundary, the hypothesis on $S$ ensures that there is a point $(\xX(q),s)$ in $S \cap {C}_{R}(\xX(p),t)$ where $f$ achieves its (positive) maximum.
Hence,
$$
f(\xX(q),s)\geq f(\xX(p),t)\geq 4N.
$$
Set $\gamma= |A_{\Sigma_s}(q)|^{-1}$ and $\sigma=R-\rho(\xX(q),s)$, so $4N \gamma \leq \sigma$. If $(\yY, \tau)\in C_{\frac{\sigma}{2}}(\xX(q),s)$, then, by the triangle inequality,
$$
\rho(\yY,\tau)\leq R-\frac{\sigma}{2}<R.
$$
Hence, 
$\sigma \leq 2 (R-\rho(\yY,\tau))$ for $(\yY,\tau)\in C_{\frac{\sigma}{2}}(\xX(q),s)$, and $C_{\frac{\sigma}{2}}(\xX(q),s)\subset C_{R}(\xX(p),t)$.
Using that $N \gamma \leq \sigma/4<\sigma/2$, we obtain
\begin{align*}
\sup_{(\xX(p'),\tau)\in S\cap C_{N \gamma}(\xX(q),s)} \frac{\sigma}{2} |A_{\Sigma_\tau}(p')| & \leq \sup_{(\xX(p'),\tau)\in S\cap C_{\frac{\sigma}{2}}(\xX(q),s)} \frac{\sigma}{2} |A_{\Sigma_\tau}(p')|\\
&\leq \sup_{S\cap C_{\frac{\sigma}{2}}(\xX(q),s)}  f \\
&\leq f(\xX(q),s) =\sigma |A_{\Sigma_s}(q)|.
\end{align*}
The lemma follows from a rearrangement of the above inequalities.
\end{proof}

\begin{proof}(of Theorem \ref{RefinedWhiteThm})
We argue by contradiction.  Suppose that the theorem was not true for some $\epsilon>0$. Then there would be a sequence of smooth mean curvature flows $S^i=\set{\Sigma_t^i}_{t\in I^i}$ in $U^i$ with $\lambda[S^i]\leq \lambda_1-\epsilon$ so that $p^i\in \Sigma^i_{t^i}$,  $C_{r^i}(\xX(p^i),t^i) \subset U^i\times I^i$ and $|A_{\Sigma^i_{t^i}}(p^i)|\geq 4 i (r^i)^{-1}$.
By Lemma \ref{BlowupLem}, there are $(\xX(q^i), s^i)\in S^i\cap C_{r^i}(\xX(p^i),t^i)$ and $\gamma^i>0$ so that
$$
\sup_{(\xX(p'),\tau)\in S^i\cap C_{i \gamma^i}(\xX(q^i),s^i)} |A_{\Sigma_\tau}(p')|\leq 2 (\gamma^{i})^{-1} = 2|A_{\Sigma_{s^i}^i}(q^i)|.
$$
Now let 
$$
\hat{S}^i=(\gamma^i)^{-1}\left({S}^i-(\xX(q^i),s^i)\right),
$$ 
that is, the flows obtained by space-time translating $(\xX(q^i),s^i)$ to the space-time origin and parabolically dilating by $(\gamma^i)^{-1}$.  Observe that these restrict to smooth flows in $C_{i}(\OO,0)$ on which the second fundamental form is bounded by $2$ and $\lambda[\hat{S}^i]=\lambda[S^i]\leq \lambda_1-\epsilon<2$.  Hence, by interior parabolic estimates \cite{EH} and the fact that the entropy is bounded strictly above by $2$,  up to passing to a subsequence, the $\hat{S}^i$ converge in $C^\infty_{loc}(\Real^{3}\times\Real)$ to a smooth mean curvature flow 
$$
\hat{S}=\set{\hat{\Sigma}_t}_{t\in (-\infty,+\infty)}
$$
in $\Real^{3}$ which satisfies $\lambda[\hat{S}]\leq \lambda_1-\epsilon<\lambda_1$, $\OO\in \hat{\Sigma}_0$,  $|A_{\hat{\Sigma}_0}(\OO)|=1$ and
$$
\sup_{t\in (-\infty,+\infty)} \sup_{\hat{\Sigma}_t} |A_{\hat{\Sigma}_t}| \leq 2.
$$
The facts that $\hat{S}$ is an eternal smooth mean curvature flow in $\Real^{3}$, $\lambda[\hat{S}]<\lambda_1$ and $|A_{\hat{\Sigma}_0}(\OO)|=1$ contradict Proposition \ref{EternalSolnProp} and so proves the theorem.
\end{proof}

\section{Proof of Theorem \ref{MainThm}}
We are now ready to use Theorem \ref{RefinedWhiteThm} in order to prove Theorem \ref{MainThm}.  
We first need the following simple consequence of Theorem \ref{RefinedWhiteThm}
\begin{lem}\label{CloseNessLem}
Given $\epsilon>0$, there is a $L=L(\epsilon)>0$ so that if $S=\set{\Sigma_t}_{t\in [0,T)}$ is a smooth mean curvature flow in $\Real^3$ of closed surfaces with $\lambda[\Sigma_0]\leq \lambda_1-\epsilon$, then for each $t\in [0, T)$,
$$
\dist_H(\Sigma_0, \Sigma_t)\leq L\left\{\begin{array}{cc} \sqrt{t} & t\in [0, \frac{T}{2}] \\ \sqrt{2 T}- \sqrt{T-t} & t\in (\frac{T}{2},T).\end{array}\right.
$$
\end{lem}
\begin{proof}
As the $\Sigma_t$ are closed, standard ODE theory ensures that there is a time-varying parametrization of the $\Sigma_t$, $\Psi: M\times [0,T)\to \Real^3$ with the property that 
$$
\frac{\partial}{\partial t}\Psi(p,t) =\mathbf{H}_{\Sigma_t}(\Psi(p,t)).
$$
Furthermore, by Huisken's monotonicity formula $\lambda[S]=\lambda[\Sigma_0]\leq \lambda_1-\epsilon$.
Hence, for $t\in (0, T)$, Theorem \ref{RefinedWhiteThm}, implies that for all $p\in M$,
$$
|A_{\Sigma_t}(\Psi(p,t))|\leq C(\epsilon)\left\{\begin{array}{cc} \frac{1}{\sqrt{t}} & t\in [0, \frac{T}{2}] \\ \frac{1}{\sqrt{T-t}} & t\in (\frac{T}{2},T).\end{array}\right.
$$
Hence, for $t\in (0,T)$ and all $p\in M$ we have
$$
\left\vert \frac{\partial}{\partial t}\Psi(p,t) \right\vert  \leq \sqrt{2} \, C(\epsilon)\left\{\begin{array}{cc} \frac{1}{\sqrt{t}} & t\in [0, \frac{T}{2}] \\ \frac{1}{\sqrt{T-t}} & t\in (\frac{T}{2},T).\end{array}\right.
$$
Hence, integrating implies that for all $p\in M$,
$$
|\Psi(p,0)-\Psi(p,t)|\leq 2\sqrt{2} \left\{\begin{array}{cc} \sqrt{t} & t\in [0, \frac{T}{2}] \\ \sqrt{2T}-\sqrt{T-t} & t\in (\frac{T}{2},T),\end{array}\right.
$$
which completes the proof with  $L=2\sqrt{2}\, C(\epsilon)$.
\end{proof}
\begin{proof} (of Theorem \ref{MainThm})
We argue by contradiction.  
That is, for some $\epsilon>0$, suppose there are closed hypersurfaces $\Sigma^i$ with $\lambda[\Sigma^i]\leq \lambda_2+\frac{1}{i}$ and so that $\dist_H(\rho \mathbb{S}^2+\yY,  \Sigma^i)\geq \rho \epsilon>0$ for all $\rho>0$ and $\yY\in \Real^3$.  Let $S^i=\set{\Sigma_t^i}_{t\in [0,T^i)}$ be the maximal smooth mean curvature flow with $\Sigma^i_0=\Sigma^i$.  By  \cite[Corollary 1.3]{BernsteinWang2}, Huisken's monotonicity formula \cite{Huisken} and the fact that $\lambda[\Sigma^i]\leq \lambda_1$, the first (and only) singularity of this flow is at $(\xX^i, T^i)$ where the flow disappears in a round point.
Now let 
$$
\hat{S}^i = (T^i)^{-1/2} (S^i -(\xX^i, T^i))=\set{\hat{\Sigma}_t^i}_{t\in [-1,0)}.
$$

Observe that $\lambda[\hat{S}^i]=\lambda[S^i]\leq \lambda_2+\frac{1}{i}$.  In particular, up to throwing out small values of $i$, by Lemma \ref{CloseNessLem}, there is a constant $\tau\in (-1,-\frac{1}{2})$, independent of $i$, so that
$$
\dist_H(\hat{\Sigma}_{-1}^i, \hat{\Sigma}_\tau^i)+\dist_H(2\mathbb{S}^2, \sqrt{-4\tau}\, \mathbb{S}^2)<\frac{\epsilon}{2}.
$$

By Theorem \ref{RefinedWhiteThm} and standard interior parabolic estimates \cite{EH}, up to passing to a subsequence, the $\hat{S}^i$ converge in $C^\infty_{loc}(\Real^{3}\times(-1,0))$ to a smooth mean curvature flow $\hat{S}$.  Clearly, $\lambda[\hat{S}]=\lambda_2$ and, by the upper semi-continuity of Gaussian density, this flow becomes singular at $(\OO,0)$.  As such, by \cite[Theorem 1.1]{BernsteinWang}, 
$$
\hat{S}=\set{\sqrt{-4t}\, \mathbb{S}^2}_{t\in (-1,0)}.
$$
In particular, we see that $\hat{\Sigma}^i_\tau \to \sqrt{-4\tau}\, \mathbb{S}^2$ in $C^\infty_{loc}(\Real^3)$.  As the $\hat{\Sigma}^i_\tau$ are connected, this implies that for $i$ sufficiently large, $\dist_H(\hat{\Sigma}^i_\tau, \sqrt{-4\tau}\, \mathbb{S}^2)<\frac{\epsilon}{4}$.  Hence, by the triangle inequality, for all $i$ sufficiently large,
$$
\dist_H(\hat{\Sigma}_{-1}^i, 2\mathbb{S}^2)< \frac{3}{4}\epsilon.
$$
That is,
$$
\dist_H(\Sigma^i_0, 2\sqrt{T^i}\, \mathbb{S}^2+\xX^i)<\frac{3}{4}\sqrt{T^i}\, \epsilon < 2\sqrt{T^i}\, \epsilon.
$$
As $\Sigma^i_0=\Sigma^i$, this contradicts our hypotheses and proves the theorem.
\end{proof}

\section{Bonnesen-style inequality}
In this section we use certain parabolic estimates, specifically Moser iteration and Schauder estimates, together with Huisken's monotonicity to prove Theorem \ref{HolderThm}. In order to state our preliminary estimates, note that for any smooth mean curvature flow $S=\set{\Sigma_t}_{t\in (-T,0)}$ in $\Real^{n+1}$, then the following function
$$
\phi_S(\xX(p),t)=2 t H_{\Sigma_t}(p)+\xX(p)\cdot \nN_{\Sigma_t}(p)
$$
is well-defined along the flow.  As observed by Smoczyk \cite[Proposition 4]{Smoczyk}, $\phi_S$ satisfies 
$$
\frac{d}{dt}\phi_S=\Delta_{\Sigma_t} \phi_S+|A_{\Sigma_t}|^2 \phi_S.
$$

\begin{prop} \label{EstProp}
If $S=\set{\Sigma_t}_{t\in (-T,0)}$ is a smooth mean curvature flow in $\Real^{n+1}$ of closed hypersurfaces and it satisfies
\begin{equation} \label{CurvAssumpEqn}
\sup_{t\in (-T,-\frac{T}{2})} (t+T)^{1/2}\sup_{\Sigma_t} |A_{\Sigma_t}|\leq C_0
\end{equation}
for some constant $C_0>0$, then there exists a $K_0=K_0(n,C_0)$ so that, for each $\tau\in (0,T/4]$,
$$
\sup_{\Sigma_{\tau-T}}\left( \tau^{-1} |\phi_S|^2+ |\nabla_{\Sigma_{\tau-T}} \phi_S|^2\right) \leq K_0 \tau^{-\frac{n}{2}-2} \int_{\frac{\tau}{2}-T}^{\tau-T} \int_{\Sigma_t} |\phi_S|^2 d\mathcal{H}^{n} dt.
$$
\end{prop}
\begin{proof}
Fix any $\tau\in (0,T/4]$. Consider the parabolically rescaled flow
$$
\hat{S}=\set{\hat{\Sigma}_t}_{t\in (-\hat{T},0)} 
$$
given by 
$$
\hat{S}=\tau^{-1/2} (S-(\OO,0)).
$$
The hypothesis \eqref{CurvAssumpEqn} implies
\begin{equation} \label{ScaleCurvAssumpEqn}
\sup_{t\in (-\hat{T},-\frac{\hat{T}}{2})} (t+\hat{T})^{1/2} \sup_{\hat{\Sigma}_t} |A_{\hat{\Sigma}_t}| \leq C_0.
\end{equation}
It follows that there exists a $\rho=\rho(C_0,n)\in (0,\frac{1}{4})$ so that for each $p_0\in \hat{\Sigma}_{1-\hat{T}}$,  $\hat{S}\cap C_{2\rho}(\xX(p_0), 1-\hat{T})$ is the graph of a function $u_{p_0}$ over $T_{p_0} \hat{\Sigma}_{1-\hat{T}}\times (1-\hat{T}-4\rho^2, 1-\hat{T}+4\rho^2)$ which satisfies the (spatial) gradient bound
$$
|D u_{p_0}|\leq \frac{1}{100}.
$$
This gradient bound ensures that the domain of $u_{p_0}$ contains
$$ 
\left(B^n_{\rho}(p_0)\cap T_{p_0}\hat{\Sigma}_{1-\hat{T}}\right) \times  (1-\hat{T}-\rho^2,1-\hat{T}+\rho^2)
$$
and we henceforth restrict to this domain. Standard interior parabolic estimates for $u_{p_0}$ (see \cite{Lieberman} or \cite{EH}) imply that on this domain
$$
|u_{p_0}|+ |D u_{p_0}| +  |D^2 u_{p_0}| +|\partial_t u_{p_0}|+|D^3 u_{p_0}|\leq M_1
$$
for some $M_1=M_1(C_0,n)>0$.
	
Consider the natural parameterization of a piece of $S$ on
$$
B^n_{\rho}(\OO)\times  (-\rho^2,\rho^2) \subset \Real^{n}\times \Real
$$
given by	
$$
\Psi_{p_0}:(x,s)\mapsto \xX(p_0)+u_{p_0}(x,s+1-\hat{T})\nN_{\hat{\Sigma}_{1-\hat{T}}}(p_0).
$$
Use these parameterizations to pull everything back to $B_{\rho}^n(\OO)\times  (-\rho^2,\rho^2)$.  One verifies that in these coordinates the pull back, $\psi$, of $\phi_{\hat{S}}$ satisfies an equation of the form
$$
\frac{\partial}{\partial s} \psi- a^{ij} \partial^2_{ij}\psi + b^i \partial_i \psi  +c \psi=0
$$
where for any $\eta=(\eta_1,\ldots,\eta_n)$,
$$
\frac{1}{2}|\eta|^2\leq a^{ij}(x,s)\eta_i \eta_j\leq |\eta|^2,
$$
and
$$
\sup_{B_{\rho}(\OO)\times  (-\rho^2,\rho^2)} \left(|b^i| +|c|+ |D a^{ij}|+|D b^i|+|D c|\right) \leq M_2
$$
for some $M_2=M_2( M_1)$. 

Thus, standard parabolic Schauder estimates (e.g., \cite[Theorem 4.9]{Lieberman}) imply that
$$
|D \psi (\OO,0)|\leq K_1' \sup_{B_{\frac{1}{2}\rho}(\OO)\times[-\frac{1}{4}\rho^2,0]} |\psi|
$$
for some $K_1'=K_1'(M_2)$. Hence, 
\begin{equation} \label{GradEstEqn}
\sup_{\hat{\Sigma}_{1-\hat{T}}} |\nabla_{\hat{\Sigma}_{1-\hat{T}} }\phi_{\hat{S}}| \leq K_1 \sup_{s\in [-\frac{1}{4}\rho^2,0]} \sup_{\hat{\Sigma}_{1-\hat{T}+s}} |\phi_{\hat{S}}|,
\end{equation}
where $K_1=K_1(K_1',M_1)$.
To complete the proof we apply the Moser iteration (e.g., \cite[Theorem 7.36]{Lieberman}) to obtain
$$
\sup_{B_{\frac{1}{2}\rho}(\OO)\times[-\frac{1}{4}\rho^2,0]} |\psi|^2 \leq K_2' \int_{-\rho^2}^0 \int_{B_{\rho}(\OO)} |\psi|^2 d\mathcal{L}^n ds
$$
for some constant $K_2'=K_2'(M_2)$, where $d\mathcal{L}^n$ is Lebesgue measure in $\Real^n$. Clearly, this implies that
\begin{equation} \label{L2EstEqn}
 \sup_{s\in [-\frac{1}{4}\rho^2,0]} \sup_{\hat{\Sigma}_{1-\hat{T}+s}} |\phi_{\hat{S}}|^2\leq  K_2 \int_{1-\hat{T}-\rho^2}^{1-\hat{T}} \int_{\hat{\Sigma}_t} |\phi_{\hat{S}}|^2 d\mathcal{{H}}^n dt
\end{equation}
for $K_2=K_2(K_2',M_1)$.

Combining \eqref{GradEstEqn} and \eqref{L2EstEqn} gives
\begin{equation} \label{ScaleEstEqn}
\sup_{\hat{\Sigma}_{1-\hat{T}}} \left(|\phi_{\hat{S}}|^2+|\nabla_{\hat{\Sigma}_{1-\hat{T}}} \phi_{\hat{S}}|^2\right) \leq K_0 \int_{\frac{1}{2}-\hat{T}}^{1-\hat{T}} \int_{\hat{\Sigma}_t} |\phi_{\hat{S}}|^2 d\mathcal{{H}}^n dt
\end{equation}
for $K_0=(1+K_1^2)K_2$.

Therefore, substituting the relation that $\phi_{\hat{S}}(\cdot\, ,\cdot)=\tau^{-1/2}\phi_S(\tau^{1/2}\cdot\, ,\tau\cdot)$ into \eqref{ScaleEstEqn}, the proposition follows immediately from changing variables.
\end{proof}

Given a closed hypersurface $\Sigma\subset \Real^{n+1}$ there always exists an $\epsilon_0=\epsilon_0(\Sigma)$ so that for any $u\in C^\infty(\Sigma)$ with $\Vert u\Vert_{C^2(\Sigma)}\leq \epsilon_0$, the normal graph exponential graph of $u$ is a closed hypersurface. For any $\epsilon<\epsilon_0(\Sigma)$, let
$$
\mathcal{N}_\epsilon(\Sigma)=\set{\Gamma: \mbox{$\Gamma$ is the normal exponential graph of a $u\in C^\infty(\Sigma)$ with $\Vert u\Vert_{C^2(\Sigma)}\leq \epsilon$} }.
$$
Recall, the following fact (cf. \cite[Lemma 2.5]{CMLoj})
\begin{lem}\label{CMlem}
For each $\alpha\in (0,1]$, there is an $\epsilon>0$ and a $\Lambda>0$ so that if $\Gamma\in\mathcal{N}_\epsilon(\sqrt{2n}\, \mathbb{S}^n)$ and $\Gamma$ is the normal exponential graph of $u$ over $\sqrt{2n}\, \mathbb{S}^n$, then
$$
\Vert u \Vert_{C^{2,\alpha}(\sqrt{2n}\, \mathbb{S}^n)} \leq \Lambda \left\Vert H_{\Gamma}-\frac{\xX\cdot \nN}{2}\right\Vert_{C^\alpha(\sqrt{2n}\, \mathbb{S}^n)}.
$$
\end{lem}

Now we are ready to prove Theorem \ref{HolderThm}.
\begin{proof}(of Theorem \ref{HolderThm})
First observe, as $\Sigma$ is closed, by rescaling and translating so $\Sigma\subset \bar{B}_1(\OO)$, we have that $\dist_H(\Sigma, \mathbb{S}^2)\leq 1$.  As such if $\lambda[\Sigma]-\lambda_2 > \delta>0$, then the claim holds as long as $K>\delta^{-1/8}$.  For this reason, we may, with out loss of generality, restrict attention to $\Sigma$ with $\lambda[\Sigma]-\lambda_2\leq \delta$. Here $\delta<\frac{1}{2}\left(\lambda_1-\lambda_2\right)$ will be determined below.
	
Let $S=\set{\Sigma_t}_{t\in [0,T)}$ be the maximal smooth mean curvature flow with $\Sigma_0=\Sigma$.  By  \cite[Corollary 1.3]{BernsteinWang2}, Huisken's monotonicity formula \cite{Huisken} and the fact that $\lambda[\Sigma]\leq\lambda_2+\delta< \lambda_1$, the first (and only) singularity of this flow is at the point $(\xX, T)$ where the flow disappears in a round point.
Now let 
$$
\hat{S} = T^{-1/2} (S -(\xX, T))=\set{\hat{\Sigma}_t}_{t\in [-1,0)}.
$$
By replacing $\Sigma$ by $\hat{\Sigma}_0$, we may, without loss of generality, assume that $S=\hat{S}$.
	
We next claim that given $\epsilon>0$, there is a $\delta>0$ so that if $\lambda[\Sigma]\leq \lambda_2+\delta$, then for all $t\in [-\frac{3}{4},-\frac{1}{4}]$, 
$$
(-t)^{-1/2}{\Sigma}_t\in \mathcal{N}_\epsilon(2\mathbb{S}^2).
$$  
To see this we argue by contradiction.  Indeed, suppose that $\Sigma^i$ are closed surfaces satisfying $\lambda[\Sigma^i]\to \lambda_2$, but for which the claim did not hold.  By Theorem \ref{RefinedWhiteThm} and standard interior parabolic estimates \cite{EH}, up to passing to a subsequence, the ${S}^i$ converge in $C^\infty_{loc}(\Real^{3}\times(-1,0))$ to a smooth mean curvature flow ${S}^\infty$.  Clearly, $\lambda[{S}^\infty]=\lambda_2$ and, by the upper semi-continuity of Gaussian density, this flow becomes singular at $(\OO,0)$.  As such, by \cite[Theorem 1.1]{BernsteinWang}, 
$$
{S}^\infty=\set{\sqrt{-4t}\, \mathbb{S}^2}_{t\in (-1,0)},
$$
and we obtain the desired contradiction.
	
Let $\epsilon>0$ and $\Lambda>0$ be the constants given by Lemma \ref{CMlem} with $\alpha=1$ and use this $\epsilon$ to select $\delta>0$ as above.  Let $\Sigma_t'=(-t)^{-1/2} \Sigma_t$ and
$$
t_0(\Sigma)=\inf\set{t\in (-1,-1/4): \Sigma_s'\in\mathcal{N}_\epsilon(2\mathbb{S}^2) \mbox{ for all $s\in [t,-1/4]$}}.
$$
Our choice of $\delta$ ensures that $t_0(\Sigma)\in [-1, -\frac{3}{4}]$.
	
If $t_0=t_0(\Sigma)>-1$, then $\Sigma_{t_0}'$ is the normal exponential graph over $2\mathbb{S}^2$ of a function $u_{t_0}$ with $C^2$ norm exactly $\epsilon$.  
Observe that
\begin{align*}
\left\Vert H_{\Sigma_{t_0}'}-\frac{\xX\cdot \nN}{2}\right\Vert_{C^{0}(\Sigma_{t_0}')} &= \left\Vert (-t_0)^{1/2} H_{\Sigma_{t_0}}-\frac{\xX\cdot \nN}{2(-t_0)^{1/2}}\right\Vert_{C^{0}(\Sigma_{t_0})} \\
&= \frac{1}{2}(-t_0)^{-1/2}  \left\Vert\phi_S|_{\Sigma_{t_0}} \right\Vert_{C^{0}(\Sigma_{t_0})}
\end{align*}
and
\begin{align*}
& \left\Vert\nabla_{\Sigma_{t_0}'}\left( H_{\Sigma_{t_0}'}-\frac{\xX\cdot \nN}{2}\right)\right\Vert_{C^{0}(\Sigma_{t_0}')} \\
= & \left\Vert (-t_0)^{1/2} \nabla_{\Sigma_{t_0}}\left( (-t_0)^{1/2} H_{\Sigma_{t_0}}-\frac{\xX\cdot \nN}{2(-t_0)^{1/2}}\right)\right\Vert_{C^{0}(\Sigma_{t_0})}\\
= & \frac{1}{2}\left\Vert\nabla_{\Sigma_{t_0}}\phi_S|_{\Sigma_{t_0}} \right\Vert_{C^{0}(\Sigma_{t_0})}.
\end{align*}
Since $t_0\leq -\frac{3}{4}$,
$$
\left\Vert H_{\Sigma_{t_0}'}-\frac{\xX\cdot \nN}{2}\right\Vert_{C^{0,1}(\Sigma_{t_0}')} \leq 2\left\Vert \phi_S|_{\Sigma_{t_0}} \right\Vert_{C^{0,1}(\Sigma_{t_0})}.
$$
By Theorem \ref{RefinedWhiteThm}, the curvature assumption \eqref{CurvAssumpEqn} holds for the flow $S$.
Hence, setting $\tau=t_0+1\leq\frac{1}{4}$, Lemma \ref{CMlem} and Proposition \ref{EstProp} imply that
\begin{align*}
\epsilon^2 &= \Vert u_{t_0} \Vert^2_{C^2(2\mathbb{S}^2)} \leq\Vert u_{t_0} \Vert^2_{C^{2,1}(2\mathbb{S}^2)} \leq \Lambda^2\left\Vert H_{\Sigma_{t_0}'}-\frac{\xX\cdot \nN}{2}\right\Vert^2_{C^{0,1}(\Sigma_{t_0}')} \\
&\leq \tilde{K}_1 \tau^{-3} \int_{ \frac{\tau}{2}-1}^{\tau-1} \int_{\Sigma_t}|\phi_S|^2 d\mathcal{H}^2 dt
\end{align*}
where $\tilde{K}_1$ depends only on $K_0$ and $\Lambda$.  
	
Next, observe that Lemma \ref{CloseNessLem} and the fact that $\lambda[\Sigma]<\frac{\lambda_1+\lambda_2}{2}<\lambda_1$ imply that there is some universal $L>0$ so that for all $p\in \Sigma_t$, $|\xX(p)|\leq 2L+3$.
Thus there is some universal $\kappa>1$ so that for all $t\in [-1, -\frac{3}{4}]$ and $p\in\Sigma_t$,
$$
\kappa^{-1} \leq (-4\pi t)^{-1} e^{\frac{|\xX(p)|^2}{4t}}\leq \kappa.
$$
Hence, 
\begin{equation} \label{T0EstEqn}
\begin{split}
\epsilon^2 &\leq \tilde{K}_1 \tau^{-3} \int_{ \frac{\tau}{2}-1}^{\tau-1} \int_{\Sigma_t} |\phi_S|^2 d\mathcal{H}^2 dt\\
&\leq \kappa\tilde{K}_1 \tau^{-3} \int_{ \frac{\tau}{2}-1}^{\tau-1} \int_{\Sigma_t} |\phi_S|^2 (-4\pi t)^{-1} e^{\frac{|\xX|^2}{4t}}d\mathcal{H}^2 dt\\	
&\leq 4\kappa\tilde{K}_1 \tau^{-3}\int_{ \frac{\tau}{2}-1}^{\tau-1} \int_{\Sigma_t} \left|H_{\Sigma_t}+\frac{\xX\cdot \nN}{2t}\right|^2 (-4\pi t)^{-1} e^{\frac{|\xX|^2}{4t}}d\mathcal{H}^2 dt\\
&\leq 4\kappa\tilde{K}_1 \tau^{-3}\left(\lambda[\Sigma]-\lambda_2\right)
\end{split}
\end{equation}
where the last inequality follows from Huisken's monotonicity formula and the definition of entropy.
This may be rewritten as
$$
\tau\leq \tilde{K}_2 \left(\lambda[\Sigma]-\lambda_2\right)^{1/3}
$$
for $\tilde{K}_2=(4\kappa\tilde{K}_1)^{1/3}\epsilon^{-2/3}$.
	
By \cite[Theorem 1.1]{BernsteinWang}, $\Sigma$ is round iff $\lambda[\Sigma]=\lambda_2$, in which the theorem holds for any $K>0$. Without loss of generality we assume $\lambda[\Sigma]>\lambda_2$.
Let 
$$
\tau_0= \tilde{K}_2 \left(\lambda[\Sigma]-\lambda_2\right)^{1/3}.
$$
By shrinking $\delta$ we may force $\tau_0<\frac{1}{8}$ and so ensure that $\tau_*=\tau_0^{3/4}>\tau_0$ and $\tau_*<\frac{1}{4}$. Clearly, 
$$
-1\leq t_0(\Sigma)<-1+\tau_*=t_*<-\frac{3}{4}.
$$  
	
By Lemma \ref{CloseNessLem}, as long as $\delta$ is sufficiently small,
$$
\dist_H(\Sigma, \Sigma_{t_*})\leq L \sqrt{\tau_*}= \tilde{K}_3  \left(\lambda[\Sigma]-\lambda_2\right)^{1/8}
$$
where $\tilde{K}_3=L\tilde{K}_2^{3/8}$.
Furthermore, by Proposition \ref{EstProp} and the fact that $t_0(\Sigma)<t_*$, a similar argument as \eqref{T0EstEqn} gives
\begin{align*}
\dist_H(\Sigma_{t_*}', 2\mathbb{S}^2)^2 & \leq \Vert u_{t_*} \Vert^2_{C^{2,1}(2\mathbb{S}^2)}\leq \Lambda^2 \left\Vert H_{\Sigma^\prime_{t_*}}-\frac{\xX\cdot \nN}{2} \right\Vert^2_{C^{0,1}(\Sigma^\prime_{t_*})}\\
&\leq \tilde{K}_1 \tau^{-3}_* \int_{ \frac{\tau_*}{2}-1}^{\tau_*-1} \int_{\Sigma_t}|\phi_S|^2 d\mathcal{H}^2 dt \\
&\leq 4 \kappa\tilde{K}_1 \tau^{-3}_* \left(\lambda[\Sigma]-\lambda_2\right)\\
&\leq 4 \kappa \tilde{K}_1 \tilde{K}_2^{-9/4} \left(\lambda[\Sigma]-\lambda_2\right)^{1/4}.
\end{align*}
Hence, as $t_*<-\frac{3}{4}$,
$$
\dist_H(\Sigma_{t_*}, 2(-t_*)^{1/2}\mathbb{S}^2) \leq  4 (\kappa \tilde{K}_1)^{1/2} \tilde{K}_2^{-9/8}\left(\lambda[\Sigma]-\lambda_2\right)^{1/8},
$$
and the result follows from the triangle inequality.
\end{proof}

\end{document}